\documentclass{article}
\usepackage{amsmath,amssymb,amsthm}
  \usepackage{paralist}
  \usepackage{graphics} %% add this and next lines if pictures should be in esp format
  \usepackage{epsfig} %For pictures: screened artwork should be set up with an 85 or 100 line screen
\usepackage{psfrag,subfigure,rotating}
\newtheorem{proposition}{Proposition}
\newtheorem{conjecture}{Conjecture}

\theoremstyle{definition}

\begin{document}

%% Place the running title of the paper with 40 letters or less in []
 %% and the full title of the paper in { }.
\title{Homoclinic Orbits of the FitzHugh-Nagumo Equation: The Singular-Limit}
\date{}
% Place all authors' names in [ ] shown as running head;
% No more than 40 letters. Leave { } empty
% Please use `and' to connect the last two names if appliable
\author{John Guckenheimer and Christian Kuehn}

\maketitle

\begin{abstract}
The FitzHugh-Nagumo equation has been investigated with a wide array of different methods in the last three decades. Recently a version of the equations with an applied current was analyzed by Champneys, Kirk, Knobloch, Oldeman and Sneyd \cite{Sneydetal} using numerical continuation methods. They obtained a complicated bifurcation diagram in parameter space featuring a C-shaped curve of homoclinic bifurcations and a U-shaped curve of Hopf bifurcations. We use techniques from multiple time-scale dynamics to understand the structures of this bifurcation diagram based on geometric singular perturbation analysis of the FitzHugh-Nagumo equation. Numerical and analytical techniques show that if the ratio of the time-scales in the FitzHugh-Nagumo equation tends to zero, then our singular limit analysis correctly represents the observed CU-structure. Geometric insight from the analysis can even be used to compute bifurcation curves which are inaccessible via continuation methods. The results of our analysis are summarized in a singular bifurcation diagram.   
\end{abstract}

\section{Introduction}
\subsection{Fast-Slow Systems}

Fast-slow systems of ordinary differential equations (ODEs) have the general form:
\begin{eqnarray}
\label{eq:fssgen}
\epsilon\dot{x}&=&\epsilon \frac{dx}{d\tau}=f(x,y,\epsilon)\\
\dot{y}&=&\frac{dy}{d\tau}=g(x,y,\epsilon)\nonumber
\end{eqnarray}
where $x\in\mathbb{R}^m$, $y\in\mathbb{R}^n$ and $0\leq \epsilon\ll 1$ represents the ratio of time scales. The functions $f$ and $g$ are assumed to be sufficiently smooth. In the singular limit $\epsilon\rightarrow 0$ the vector field (\ref{eq:fssgen}) becomes a differential-algebraic equation. The algebraic constraint $f=0$ defines the critical manifold $C_0=\{(x,y)\in\mathbb{R}^m\times\mathbb{R}^n:f(x,y,0)=0\}$. Where $D_xf(p)$ is nonsingular, the implicit function theorem implies that there exists a map $h(x)=y$ parametrizing $C_0$ as a graph. This yields the implicitly defined vector field $\dot{y}=g(h(y),y,0)$ on $C_0$ called the slow flow.\\

We can change (\ref{eq:fssgen}) to the fast time scale $t=\tau/\epsilon$ and let $\epsilon\rightarrow 0$ to obtain the second possible singular limit system
\begin{eqnarray}
\label{eq:fssfss}
x'&=&\frac{dx}{dt}=f(x,y,0)\\
y'&=&\frac{dy}{dt}=0\nonumber
\end{eqnarray}
We call the vector field (\ref{eq:fssfss}) parametrized by the slow variables $y$ the fast subsystem or the layer equations. The central idea of singular perturbation analysis is to use information about the fast subsystem and the slow flow to understand the full system (\ref{eq:fssgen}). One of the main tools is Fenichel's Theorem (see \cite{Fenichel1,Fenichel2,Fenichel3,Fenichel4}). It states that for every $\epsilon$ sufficiently small and $C_0$ normally hyperbolic there exists a family of invariant manifolds $C_\epsilon$ for the flow (\ref{eq:fssgen}). The manifolds are at a distance $O(\epsilon)$ from $C_0$ and the flows on them converge to the slow flow on $C_0$ as $\epsilon\rightarrow 0$. Points $p\in C_0$ where $D_xf(p)$ is singular are referred to as fold points\footnote{The projection of $C_0$ onto the $x$ coordinates may have more degenerate singularities than fold singularities at some of these points.}.\\

Beyond Fenichel's Theorem many other techniques have been developed. More detailed introductions and results can be found in \cite{ArnoldEncy,Jones,GuckenheimerNDC} from a geometric viewpoint. Asymptotic methods are developed in \cite{MisRoz,Grasman} whereas ideas from nonstandard analysis are introduced in \cite{DienerDiener}. While the theory is well developed for two-dimensional fast-slow systems, higher-dimensional fast-slow systems are an active area of current research. In the following we shall focus on the FitzHugh-Nagumo equation viewed as a three-dimensional fast-slow system.

\subsection{The FitzHugh-Nagumo Equation}
\label{sec:fhn}

The FitzHugh-Nagumo equation is a simplification of the Hodgin-Huxley model for an electric potential of a nerve axon \cite{HodginHuxley}. The first version was developed by FitzHugh \cite{FitzHugh} and is a two-dimensional system of ODEs:
\begin{eqnarray}
\label{eq:fh}
\epsilon \dot{u}&=&v-\frac{u^3}{3}+u+p\\
\dot{v}&=&-\frac1s(v+\gamma u-a)\nonumber
\end{eqnarray}
A detailed summary of the bifurcations of (\ref{eq:fh}) can be found in \cite{GGR}. Nagumo et al. \cite{Nagumo} studied a related equation that adds a diffusion term for the conduction process of action potentials along nerves:
\begin{equation}
\label{eq:fhn_original}
\left\{
\begin{array}{l}
u_\tau=\delta u_{xx}+f_a(u)-w+p \\
w_\tau=\epsilon(u-\gamma w)
\end{array}
\right.
\end{equation}
where $f_a(u)=u(u-a)(1-u)$ and $p,\gamma,\delta$ and $a$ are parameters. A good introduction to the derivation and problems associated with (\ref{eq:fhn_original}) can be found in \cite{Hastings}. Suppose we assume a traveling wave solution to (\ref{eq:fhn_original}) and set $u(x,\tau)=u(x+s\tau)=u(t)$ and $w(x,\tau)=w(x+s\tau)=w(t)$, where $s$ represents the wave speed. By the chain rule we get $u_\tau=su'$, $u_{xx}=u''$ and $w_\tau=sw'$. Set $v=u'$ and substitute into (\ref{eq:fhn_original}) to obtain the system:
\begin{eqnarray}
\label{eq:fhn_temp}
u'&=&v \nonumber\\
v'&=&\frac1\delta(sv-f_a(u)+w-p)\\
w'&=&\frac{\epsilon}{s}(u-\gamma w)\nonumber
\end{eqnarray}
System~\eqref{eq:fhn_temp} is the FitzHugh-Nagumo equation studied in this paper. Observe that a homoclinic orbit of (\ref{eq:fhn_temp}) corresponds to a traveling pulse solution of (\ref{eq:fhn_original}). These solutions are of special importance in neuroscience \cite{Hastings} and have been analyzed using several different methods. For example, it has been proved that (\ref{eq:fhn_temp}) admits homoclinic orbits \cite{Hastings1,Carpenter} for small wave speeds (``slow waves'') and large wave speeds (``fast waves''). Fast waves are stable \cite{JonesFHN} and slow waves are unstable \cite{Flores}. It has been shown that double-pulse homoclinic orbits \cite{EvansFenichelFeroe} are possible. If (\ref{eq:fhn_temp}) has two equilibrium points and heteroclinic connections exist, bifurcation from a twisted double heteroclinic connection implies the existence of multi-pulse traveling front and back waves \cite{DengFHN}. These results are based on the assumption of certain parameter ranges for which we refer to the original papers. Geometric singular perturbation theory has been used successfully to analyze (\ref{eq:fhn_temp}). In \cite{JonesKopellLanger} the fast pulse is constructed using the exchange lemma \cite{JonesKaperKopell,JonesKopell,Brunovsky}. The exchange lemma has also been used to prove the existence of a codimension two connection between fast and slow waves in ($s,\epsilon,a$)-parameter space \cite{KSS1997}. An extension of Fenichel's theorem and Melnikov's method can be employed to prove the existence of heteroclinic connections for parameter regimes of (\ref{eq:fhn_temp}) with two fixed points \cite{Szmolyan1}. The general theory of relaxation oscillations in fast-slow systems applies to (\ref{eq:fhn_temp}) (see e.g. \cite{MisRoz,GuckenheimerBoRO}) as does - at least partially - the theory of canards (see e.g. \cite{Wechselberger,Dumortier1,Eckhaus,KruSzm2}).\\

The equations (\ref{eq:fhn_temp}) have been analyzed numerically by Champneys, Kirk, Knobloch, Oldeman and Sneyd \cite{Sneydetal} using the numerical bifurcation software AUTO \cite{Doedel_AUTO1997,Doedel_AUTO2000}. They considered the following parameter values:
\begin{equation*}
\gamma=1,\qquad a=\frac{1}{10},\qquad \delta=5
\end{equation*}
We shall fix those values to allow comparison of our results with theirs. Hence we also write $f_{1/10}(u)=f(u)$. Changing from the fast time $t$ to the slow time $\tau$ and relabeling variables $x_1=u$, $x_2=v$ and $y=w$ we get:
\begin{eqnarray}
\label{eq:fhn}
\epsilon\dot{x}_1&=& x_2\nonumber\\
\epsilon\dot{x}_2&=&\frac15 (sx_2-x_1(x_1-1)(\frac{1}{10}-x_1)+y-p)=\frac15 (sx_2-f(x_1)+y-p)\\
\dot{y}&=&\frac{1}{s} (x_1-y) \nonumber 
\end{eqnarray}
From now on we refer to (\ref{eq:fhn}) as ``the'' FitzHugh-Nagumo equation. Investigating bifurcations in the ($p,s$) parameter space one finds C-shaped curves of homoclinic orbits and a U-shaped curve of Hopf bifurcations; see Figure \ref{fig:cusystem}. Only part of the bifurcation diagram is shown in Figure \ref{fig:cusystem}. There is another curve of homoclinic bifurcations on the right side of the U-shaped Hopf curve. Since (\ref{eq:fhn}) has the symmetry
\begin{equation}
\label{eq:fhnsym}
x_1\rightarrow\frac{11}{15}-x_1, \quad x_2\rightarrow \frac{11}{15}-x_2,\quad y\rightarrow -y,\quad p\rightarrow \frac{11}{15}\left(1-\frac{33}{225}\right)-p
\end{equation}
we shall examine only the left side of the U-curve. The homoclinic C-curve is difficult to compute numerically by continuation methods using AUTO \cite{Doedel_AUTO1997,Doedel_AUTO2000} or MatCont \cite{MatCont}. The computations seem infeasible for small values of $\epsilon\leq 10^{-3}$. Furthermore multi-pulse homoclinic orbits can exist very close to single pulse ones and distinguishing between them must necessarily encounter problems with numerical precision \cite{Sneydetal}. The Hopf curve and the bifurcations of limit cycles shown in Figure \ref{fig:cusystem} have been computed using MatCont. The curve of homoclinic bifurcations has been computed by a new method to be described in Section \ref{sec:hetfull}.\\ 

\begin{figure}[htbp]
\centering
  \includegraphics[width=0.85\textwidth]{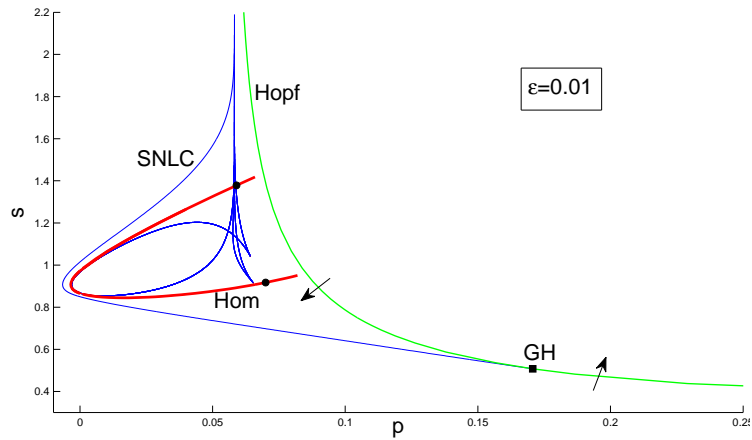} 
 \caption{Bifurcation diagram of (\ref{eq:fhn}). Hopf bifurcations are shown in green, saddle-node of limit cycles (SNLC) are shown in blue and GH indicates a generalized Hopf (or Bautin) bifurcation. The arrows indicate the side on which periodic orbits are generated at the Hopf bifurcation. The red curve shows (possible) homoclinic orbits; in fact, homoclinic orbits only exist to the left of the two black dots (see Section \ref{sec:hetfull}). Only part of the parameter space is shown because of the symmetry (\ref{eq:fhnsym}). The homoclinic curve has been thickened to indicate that multipulse homoclinic orbits exist very close to single pulse ones (see \cite{EvansFenichelFeroe}).}
 \label{fig:cusystem}
 \end{figure}

Since the bifurcation structure shown in Figure \ref{fig:cusystem} was also observed for other excitable systems, Champneys et al.~\cite{Sneydetal} introduced the term CU-system. Bifurcation analysis from the viewpoint of geometric singular perturbation theory has been carried out for examples with one fast and two slow variables \cite{GuckFvdP1,GuckFvdP2,GuckFT,MilikSzmolyan}. Since the FitzHugh-Nagumo equation has one slow and two fast variables, the situation is quite different and new techniques have to be developed. Our main goal is to show that many features of the complicated 2-parameter bifurcation diagram shown in Figure \ref{fig:cusystem} can be derived with a combination of techniques from singular perturbation theory, bifurcation theory and robust numerical methods. We accurately locate where the system has canards and determine the orbit structure of the homoclinic and periodic orbits associated to the C-shaped and U-shaped bifurcation curves, without computing the canards themselves. We demonstrate that the basic CU-structure of the system can be computed with elementary methods that do not use continuation methods based on collocation. The analysis of the slow and fast subsystems yields a ``singular bifurcation diagram'' to which the basic CU structure in Figure \ref{fig:cusystem} converges as $\epsilon\rightarrow 0$.\\

\textit{Remark:} We have also investigated the termination mechanism of the C-shaped homoclinic curve described in \cite{Sneydetal}. Champneys et al. observed that the homoclinic curve does not reach the U-shaped Hopf curve but turns around and folds back close to itself. We compute accurate approximations of the homoclinic orbits for smaller values $\epsilon$ than seems possible with AUTO in this region. One aspect of our analysis is a new algorithm for computing invariant slow manifolds of saddle type in the full system. This work will be described elsewhere.

\section{The Singular Limit}

The first step in our analysis is to investigate the slow and fast subsystems separately. Let $\epsilon\rightarrow 0$ in (\ref{eq:fhn}); this yields two algebraic constraints that define the critical manifold:
\begin{equation*}
C_0=\left\{(x_1,x_2,y)\in\mathbb{R}^3:x_2=0\quad y=x_1(x_1-1)(\frac{1}{10}-x_1)+p=c(x_1)\right\}
\end{equation*}
Therefore $C_0$ is a cubic curve in the coordinate plane $x_2=0$. The parameter $p$ moves the cubic up and down inside this plane. The critical points of the cubic are solutions of $c'(x_1)=0$ and are given by: 
\begin{equation*}
x_{1,\pm}=\frac{1}{30}\left(11\pm\sqrt{91}\right)\qquad \text{or numerically:} \quad x_{1,+}\approx 0.6846, \quad x_{1,-}\approx 0.0487
\end{equation*}
The points $x_{1,\pm}$ are fold points with $|c''(x_{1,\pm})|\neq 0$ since $C_0$ is a cubic polynomial with distinct critical points. The fold points divide $C_0$ into three segments 
\begin{equation*}
C_l=\{x_1<x_{1,-}\}\cap C_0, \quad C_m=\{x_{1,-}\leq x_1\leq x_{1,+}\}\cap C_0,\quad C_r=\{x_{1,+}<x_1\}\cap C_0
\end{equation*}
We denote the associated slow manifolds by $C_{l,\epsilon}$, $C_{m,\epsilon}$ and $C_{r,\epsilon}$. There are two possibilities to obtain the slow flow. One way is to solve $c(x_1)=y$ for $x_1$ and substitute the result into the equation $\dot{y}=\frac1s (x_1-y)$. Alternatively differentiating $y=c(x_1)$ implicitly with respect to $\tau$ yields $\dot{y}=\dot{x}_1c'(x_1)$ and therefore
\begin{equation}
\label{eq:sf}
\frac1s (x_1-y)=\dot{x}_1c'(x_1) \qquad \Rightarrow  \qquad \dot{x}_1=\frac{1}{sc'(x_1)}(x_1-c(x_1))
\end{equation}
One can view this as a projection of the slow flow, which is constrained to the critical manifold in $\mathbb{R}^3$, onto the $x_1$-axis. Observe that the slow flow is singular at the fold points. Direct computation shows that the fixed point problem $x_1=c(x_1)$ has only a single real solution. This implies that the critical manifold intersects the diagonal $y=x_1$ only in a single point $x_1^*$ which is the unique equilibrium of the slow flow (\ref{eq:sf}). Observe that $q=(x_1^*,0,x_1^*)$ is also the unique equilibrium of the full system (\ref{eq:fhn}) and depends on $p$. Increasing $p$ moves the equilibrium from left to right on the critical manifold. The easiest practical way to determine the direction of the slow flow on $C_0$ is to look at the sign of $(x_1-y)$. The situation is illustrated in Figure \ref{fig:slowflow}.\\ 

\begin{figure}[htbp]
	\centering
		\includegraphics[width=0.8\textwidth]{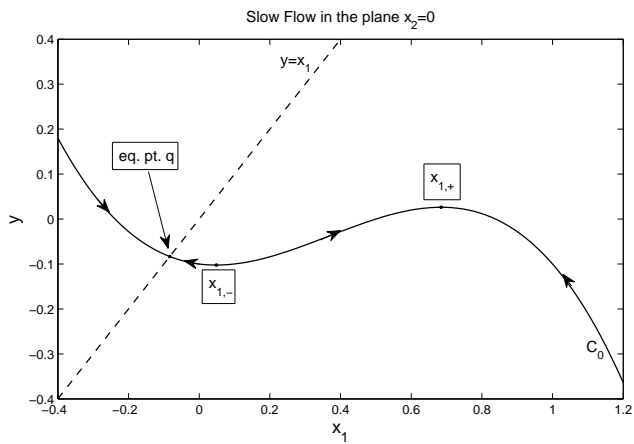}
	\caption{Sketch of the slow flow on the critical manifold $C_0$}
	\label{fig:slowflow}
\end{figure} 

\subsection{The Slow Flow}

We are interested in the bifurcations of the slow flow depending on the parameter $p$. The bifurcations occur when $x_1^*$ passes through the fold points. The values of $p$ can simply be found by solving the equations $c'(x_1)=0$ and $c(x_1)-x_1=0$ simultaneously. The result is:
\begin{equation*}
p_-\approx 0.0511 \qquad  \text{and} \qquad p_+\approx 0.5584
\end{equation*}
where the subscripts indicate the fold point at which each equilibrium is located.\\ 

The singular time-rescaling $\bar{\tau}=sc'(x_1)/\tau$ of the slow flow yields the desingularized slow flow
\begin{equation}
\label{eq:dssf}
\frac{dx_1}{d\bar{\tau}}=x_1-c(x_1)=x_1+\frac{x_1}{10}\left(x_1-1\right)\left(10x_1-1\right)-p
\end{equation}
Time is reversed by this rescaling on $C_l$ and $C_r$ since $s>0$ and $c'(x_1)$ is negative on these branches. The desingularized slow flow (\ref{eq:dssf}) is smooth and has no bifurcations as $p$ is varied. 

\subsection{The Fast Subsystem}
\label{sec:hetsub}

The key component of the fast-slow analysis for the FitzHugh-Nagumo equation is the two-dimensional fast subsystem
\begin{eqnarray}
\label{eq:fss}
x_1'&=&x_2 \nonumber\\
x_2'&=&\frac15 (sx_2-x_1(x_1-1)(\frac{1}{10}-x_1)+y-p)
\end{eqnarray}
where $p\geq 0$, $s\geq 0$ are parameters and $y$ is fixed. Since $y$ and $p$ have the same effect as bifurcation parameters we set $p-y=\bar{p}$. We consider several fixed y-values and the effect of varying $p$ (cf. Section \ref{sec:hetfull}) in each case. There are either one, two or three equilibrium points for (\ref{eq:fss}). Equilibrium points satisfy $x_2=0$ and lie on the critical manifold, i.e. we have to solve
\begin{equation}
\label{eq:eqfss}
0=x_{1}(x_{1}-1)(\frac{1}{10}-x_{1})+\bar{p}
\end{equation}
for $x_1$. We find that there are three equilibria for approximately $\bar{p}_l=-0.1262<\bar{p}<0.0024=\bar{p}_r$, two equilibria on the boundary of this $p$ interval and one equilibrium otherwise. The Jacobian of (\ref{eq:fss}) at an equilibrium is
\begin{equation*}
A(x_1)=\left(\begin{array}{cc}
0& 1\\
\frac{1}{50}\left(1-22x_1+30x_1^2\right) &\frac{s}{5}\\
\end{array}\right)
\end{equation*}
Direct calculation yields that for $p\not\in[\bar{p}_l,\bar{p}_r]$ the single equilibrium is a saddle. In the case of three equilibria, we have a source that lies between two saddles. Note that this also describes the stability of the three branches of the critical manifold $C_l$, $C_m$ and $C_r$. For $s>0$ the matrix $A$ is singular of rank 1 if and only if $30x_1^2-22x_1+1=0$ which occurs for the fold points $x_{1,\pm}$. Hence the equilibria of the fast subsystem undergo a fold (or saddle-node) bifurcation once they approach the fold points of the critical manifold. This happens for parameter values $\bar{p}_l$ and $\bar{p}_r$. Note that by symmetry we can reduce to studying a single fold point. In the limit $s=0$ (corresponding to the case of a ``standing wave'') the saddle-node bifurcation point becomes more degenerate with $A(x_1)$ nilpotent.\\ 

Our next goal is to investigate global bifurcations of (\ref{eq:fss}); we start with homoclinic orbits. For $s=0$ it is easy to see that (\ref{eq:fss}) is a Hamiltonian system:
\begin{eqnarray}
\label{eq:fss2}
x_1'&=&\frac{\partial H}{\partial x_2}=x_2\nonumber\\
x_2'&=&-\frac{\partial H}{\partial x_1}=\frac15 (-x_1(x_1-1)(\frac{1}{10}-x_1)-\bar{p})
\end{eqnarray}
with Hamiltonian function
\begin{equation}
\label{eq:ham}
H(x_1,x_2)=\frac12 x_2^2-\frac{(x_1)^2}{100}+\frac{11(x_1)^3}{150}-\frac{(x_1)^4}{20}+\frac{x_1\bar{p}}{5}
\end{equation}
We will use this Hamiltonian formulation later on to describe the geometry of homoclinic orbits for slow wave speeds. Assume that $\bar{p}$ is chosen so that 
(\ref{eq:fss2}) has a homoclinic orbit $x_0(t)$. We are interested in perturbations with $s> 0$ and note that in this case the divergence of (\ref{eq:fss}) is $s$. Hence the vector field is area expanding everywhere. The homoclinic orbit breaks for $s>0$ and no periodic orbits are created. Note that this scenario does not apply to the full three-dimensional system as the equilibrium $q$ has a pair of complex conjugate eigenvalues so that a Shilnikov scenario can occur. This illustrates that the singular limit can be used to help locate homoclinic orbits of the full system, but that some characteristics of these orbits change in the singular limit.\\

We are interested next in finding curves in $(\bar{p},s)$-parameter space that represent heteroclinic connections of the fast subsystem. The main motivation is the decomposition of trajectories in the full system into slow and fast segments. Concatenating fast heteroclinic segments and slow flow segments can yield homoclinic orbits of the full system \cite{Hastings,Carpenter,JonesKopellLanger,KSS1997}. We describe a numerical strategy to detect heteroclinic connections in the fast subsystem and continue them in parameter space. Suppose that $\bar{p}\in(\bar{p}_l,\bar{p}_r)$ so that (\ref{eq:fss}) has three hyperbolic equilibrium points $x_l$, $x_m$ and $x_r$. We denote by $W^u(x_{l})$ the unstable and by $W^s(x_{l})$ the stable manifold of $x_l$. The same notation is also used for $x_r$ and tangent spaces to $W^s(.)$ and $W^u(.)$ are denoted by $T^s(.)$ and $T^u(.)$. Recall that $x_m$ is a source and shall not be of interest to us for now. Define the  cross section $\Sigma$ by
\begin{equation*}
\Sigma=\{(x_1,x_2)\in\mathbb{R}^2:x_1=\frac{x_l+x_r}{2}\}.
\end{equation*}
We use forward integration of initial conditions in $T^u(x_l)$ and backward integration of initial conditions in $T^s(x_r)$ to obtain trajectories $\gamma^+$ and $\gamma^-$ respectively. We calculate their intersection with $\Sigma$ and define
\begin{equation*}
\gamma_l(\bar{p},s):=\gamma^+\cap \Sigma, \qquad \gamma_r(\bar{p},s):=\gamma^-\cap \Sigma
\end{equation*}
We compute the functions $\gamma_l$ and $\gamma_r$ for different parameter values of $(\bar{p},s)$ numerically. Heteroclinic connections occur at zeros of the function
\begin{equation*}
h(\bar{p},s):=\gamma_l(\bar{p},s)-\gamma_r(\bar{p},s)
\end{equation*}
Once we find a parameter pair $(\bar{p}_0,s_0)$ such that $h(\bar{p}_0,s_0)=0$, these parameters can be continued along a curve of heteroclinic connections in $(\bar{p},s)$ parameter space by solving the root-finding problem $h(\bar{p}_0+\delta_1,s_0+\delta_2)=0$ for either $\delta_1$ or $\delta_2$ fixed and small. We use this method later for different fixed values of $y$ to compute heteroclinic connections in the fast subsystem in $(p,s)$ parameter space. The results of these computations are illustrated in Figure \ref{fig:het1}. There are two distinct branches in Figure \ref{fig:het1}. The branches are asymptotic to $\bar{p}_l$ and $\bar{p}_r$ and approximately form a ``$V$''. From Figure \ref{fig:het1} we conjecture that there exists a double heteroclinic orbit for $\bar{p}\approx-0.0622$.\\
 
\begin{figure}[htbp]
\centering
\psfrag{p}{$\bar{p}$}
\psfrag{s}{$s$}
  \includegraphics[width=0.5\textwidth]{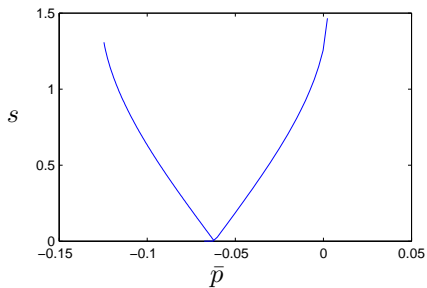} 
 \caption{Heteroclinic connections for equation (\ref{eq:fss}) in parameter space.}
 \label{fig:het1}
 \end{figure} 

\textit{Remarks}: If we fix $p=0$ our initial change of variable becomes $-y=\bar{p}$ and our results for heteroclinic connections are for the FitzHugh-Nagumo equation without an applied current. In this situation it has been shown that the heteroclinic connections of the fast subsystem can be used to prove the existence of homoclinic orbits to the unique saddle equilibrium $(0,0,0)$ (cf. \cite{JonesKopellLanger}). Note that the existence of the heteroclinics in the fast subsystem was proved in a special case analytically \cite{AronsonWeinberger} but Figure \ref{fig:het1} is - to the best of our knowledge - the first explicit computation of where fast subsystem heteroclinics are located. The paper \cite{KrLinMethod} develops a method for finding heteroclinic connections by the same basic approach we used, i.e. defining a codimension one hyperplane $H$ that separates equilibrium points.\\

Figure \ref{fig:het1} suggests that there exists a double heteroclinic connection for $s=0$. Observe that the Hamiltonian in our case is $H(x_1,x_2)=\frac{(x_2)^2}{2}+V(x_1)$ where the function $V(x_1)$ is:
\begin{equation*}
V(x_1)=\frac{px_1}{5}-\frac{(x_1)^2}{100}+\frac{11(x_1)^3}{150}-\frac{(x_1)^4}{20}
\end{equation*}
The solution curves of (\ref{eq:fss2}) are given by $x_2=\pm\sqrt{2(\text{const. }-V(x_1))}$. The structure of the solution curves entails symmetry under reflection about the $x_1$-axis. Suppose $\bar{p}\in[\bar{p}_l,\bar{p}_r]$ and recall that we denoted the two saddle points of (\ref{eq:fss}) by $x_l$ and $x_r$ and that their location depends on $\bar{p}$. Therefore, we conclude that the two saddles $x_l$ and $x_r$ must have a heteroclinic connection if they lie on the same energy level, i.e. they satisfy $V(x_l)-V(x_r)=0$. This equation can be solved numerically to very high accuracy. 

\begin{proposition}
\label{prop:doublehet}
The fast subsystem of the FitzHugh-Nagumo equation for $s=0$ has a double heteroclinic connection for $\bar{p}=\bar{p}^*\approx -0.0619259$. Given a particular value $y=y_0$ there exists a double heteroclinic connection for $p=\bar{p}^*+y_0$ in the fast subsystem lying in the plane $y=y_0$.
\end{proposition}

\subsection{Two Slow Variables, One Fast Variable}
\label{sec:2slow1fast}

From continuation of periodic orbits in the full system - to be described in Section \ref{sec:hopf} - we observe that near the U-shaped curve of Hopf bifurcations the $x_2$-coordinate is a faster variable than $x_1$. In particular, the small periodic orbits generated in the Hopf bifurcation lie almost in the plane $x_2=0$. Hence to analyze this region we set $\bar{x}_2=x_2/\epsilon$ to transform the FitzHugh-Nagumo equation (\ref{eq:fhn}) into a system with 2 slow and 1 fast variable:
\begin{eqnarray}
\label{eq:fhnscale}
\dot{x}_1&=& \bar{x}_2\nonumber\\
\epsilon^2 \dot{\bar{x}}_2&=&\frac15 (s\epsilon\bar{x}_2-x_1(x_1-1)(\frac{1}{10}-x_1)+y-p)\\
\dot{y}&=&\frac{1}{s} (x_1-y) \nonumber 
\end{eqnarray} 
Note that \eqref{eq:fhnscale} corresponds to the FitzHugh-Nagumo equation in the form (cf. \eqref{eq:fhn_original}):
\begin{equation}
\label{eq:fhn_original_small_diff}
\left\{
\begin{array}{l}
u_\tau=5\epsilon^2 u_{xx}+f(u)-w+p \\
w_\tau=\epsilon(u- w)
\end{array}
\right.
\end{equation}
Therefore the transformation $\bar{x}_2=x_2/\epsilon$ can be viewed as a rescaling of the diffusion strength by $\epsilon^2$. We introduce a new independent small parameter $\bar{\delta}=\epsilon^2$ and then let $\bar{\delta}=\epsilon^2\rightarrow 0$. This assumes that $O(\epsilon)$ terms do not vanish in this limit, yielding the diffusion free system. Then the slow manifold $S_0$ of (\ref{eq:fhnscale}) is:
\begin{equation}
\label{eq:sm2d}
S_0=\left\{ (x_1,\bar{x}_2,y) \in\mathbb{R}^3 : \bar{x}_2=\frac{1}{s\epsilon}\left(f(x_1)-y+p\right)\right\}
\end{equation}
\begin{proposition}
\label{prop:2slow1fast}
Following time rescaling by $s$, the slow flow of system~\eqref{eq:fhnscale} on $S_0$ in the variables $(x_1,y)$ is given by
\begin{eqnarray}
\label{eq:red1}
\epsilon \dot{x}_1&=&f(x_1)-y+p\nonumber\\
\dot{y}&=&x_1-y
\end{eqnarray}
In the variables $(x_1,\bar{x}_2)$ the vector field~\eqref{eq:red1} becomes
\begin{eqnarray}
\label{eq:red2}
\dot{x}_1&=& \bar{x}_2\nonumber\\
\epsilon\dot{\bar{x}}_2&=&-\frac{1}{s^2}\left(x_1-f(x_1)-p\right)+\frac{\bar{x}_2}{s}\left(f'(x_1)-\epsilon \right)
\end{eqnarray}
\end{proposition}

\textit{Remark:} The reduction to equations \eqref{eq:red1}-\eqref{eq:red2} suggests that \eqref{eq:fhnscale} is a three time-scale system. Note however that \eqref{eq:fhnscale} is not given in the three time-scale form $(\epsilon^2\dot{z}_1,\epsilon\dot{z}_2,\dot{z}_3)=(h_1(z),h_2(z),h_3(z))$ for $z=(z_1,z_2,z_3)\in\mathbb{R}^3$ and $h_i:\mathbb{R}^3\rightarrow \mathbb{R}$ $(i=1,2,3)$. The time-scale separation in \eqref{eq:red1}-\eqref{eq:red2} results from the singular $1/\epsilon$ dependence of the critical manifold $S_0$; see \eqref{eq:sm2d}.

\begin{proof}(of Proposition \ref{prop:2slow1fast})
Use the defining equation for the slow manifold (\ref{eq:sm2d}) and substitute it into $\dot{x}_1=\bar{x}_2$. A rescaling of time by $t\rightarrow st$ under the assumption that $s>0$ yields the result (\ref{eq:red1}). To derive (\ref{eq:red2}) differentiate the defining equation of $S_0$ with respect to time:
\begin{equation*}
\dot{\bar{x}}_2=\frac{1}{s\epsilon}\left(\dot{x}_1f'(x_1)-\dot{y}\right)=\frac{1}{s\epsilon}\left(\bar{x}_2f'(x_1)-\dot{y}\right)
\end{equation*}
The equations $\dot{y}=\frac{1}{s}(x_1-y)$ and $y=-s\epsilon\bar{x}_2+f(x_1)+p$
yield the equations (\ref{eq:red2}).
\end{proof}
Before we start with the analysis of (\ref{eq:red1}) we note that detailed bifurcation calculations for \eqref{eq:red1} exist. For example, Rocsoreanu et al. \cite{GGR} give a detailed overview on the FitzHugh equation \eqref{eq:red1} and collect many relevant references. Therefore we shall only state the relevant bifurcation results and focus on the fast-slow structure and canards. Equation \eqref{eq:red1} has a critical manifold given by $y=f(x_1)+p=c(x_1)$ which coincides with the critical manifold of the full FitzHugh-Nagumo system (\ref{eq:fhn}). Formally it is located in $\mathbb{R}^2$ but we still denote it by $C_0$. Recall that the fold points are located at 
\begin{equation*}
x_{1,\pm}=\frac{1}{30}\left(11\pm\sqrt{91}\right)\qquad \text{or numerically:} \quad x_{1,+}\approx 0.6846, \quad x_{1,-}\approx 0.0487
\end{equation*}
Also recall that the y-nullcline passes through the fold points at:
\begin{equation*}
p_-\approx 0.0511 \qquad  \text{and} \qquad p_+\approx 0.5584
\end{equation*}
We easily find that supercritical Hopf bifurcations are located at the values
\begin{equation}
\label{eq:Hopfred}
p_{H,\pm}(\epsilon)=\frac{2057}{6750} \pm \sqrt{\frac{11728171}{182250000}-\frac{359 \epsilon }{1350}+\frac{509 \epsilon ^2}{2700}-\frac{\epsilon ^3}{27}}
\end{equation}
For the case $\epsilon=0.01$ we get $p_{H,-}(0.01)\approx 0.05632$ and $p_{H,+}(0.01)\approx 0.55316$. The periodic orbits generated in the Hopf bifurcations exist for $p\in (p_{H,-},p_{H,+})$. Observe also that $p_{H,\pm}(0)=p_\pm$; so the Hopf bifurcations of (\ref{eq:red1}) coincide in the singular limit with the fold bifurcations in the one-dimensional slow flow (\ref{eq:sf}). We are also interested in canards in the system and calculate a first order asymptotic expansion for the location of the maximal canard in (\ref{eq:red1}) following \cite{KruSzm3}; recall that trajectories lying in the intersection of attracting and repelling slow manifolds are called maximal canards. We restrict to canards near the fold point $(x_{1,-},c(x_{1,-}))$.

\begin{proposition}
Near the fold point $(x_{1,-},c(x_{1,-}))$ the maximal canard in $(p,\epsilon)$ parameter space is given by:
\begin{equation*}
p(\epsilon)=x_{1,-}-c(x_{1,-})+\frac58 \epsilon+O(\epsilon^{3/2})
\end{equation*}
\end{proposition}
\begin{proof}
Let $\bar{y}=y-p$ and consider the shifts 
\begin{equation*}
x_1\rightarrow x_1+x_{1,-},\quad \bar{y}\rightarrow \bar{y}+c(x_{1,-}), \quad p\rightarrow p+x_{1,-}-c(x_{1,-}) 
\end{equation*}
to translate the equilibrium of (\ref{eq:red1}) to the origin when $p=0$. This gives
\begin{eqnarray}
\label{eq:red1sf}
x_1'&=&x_1^2\left(\frac{\sqrt{91}}{10}-x_1\right)-\bar{y}=\bar{f}(x_1,\bar{y})\nonumber\\
y'&=&\epsilon(x_1-\bar{y}-p)=\epsilon (\bar{g}(x_1,\bar{y})-p)
\end{eqnarray}
Now apply Theorem 3.1 in \cite{KruSzm3} to find that the maximal canard of (\ref{eq:red1sf}) is given by:
\begin{equation*}
p(\epsilon)=\frac{5}{8}\epsilon+O(\epsilon^{3/2})
\end{equation*} 
Shifting the parameter $p$ back to the original coordinates yields the result.
\end{proof}
If we substitute $\epsilon=0.01$ in the previous asymptotic result and neglect terms of order $O(\epsilon^{3/2})$ then the maximal canard is predicted to occur for $p\approx 0.05731$ which is right after the first supercritical Hopf bifurcation at $p_{H,-}\approx 0.05632$. Therefore we expect that there exist canard orbits evolving along the middle branch of the critical manifold $C_{m,0.01}$ in the full FitzHugh-Nagumo equation. Maximal canards are part of a process generally referred to as canard explosion \cite{DRvdP,KruSzm2,Diener}. In this situation the small periodic orbits generated in the Hopf bifurcation at $p=p_{H,-}$ undergo a transition to relaxation oscillations within a very small interval in parameter space. A variational integral determines whether the canards are stable \cite{KruSzm2,GuckenheimerBoRO}. 

\begin{proposition}
\label{prop:stablecanards}
The canard cycles generated near the maximal canard point in parameter space for equation (\ref{eq:red1}) are stable.
\end{proposition} 

\begin{proof}
Consider the differential equation (\ref{eq:red1}) in its transformed form (\ref{eq:red1sf}). Obviously this will not affect the stability analysis of any limit cycles. Let $x_l(h)$ and $x_m(h)$ denote the two smallest $x_1$-coordinates of the intersection between
\begin{equation*}
\bar{C}_0:=\{(x_1,\bar{y})\in\mathbb{R}^2:\bar{y}=\frac{\sqrt{91}}{10}x_1^2-x_1^3=\phi(x_1)\}
\end{equation*}
and the line $\bar{y}=h$. Geometrically $x_l$ represents a point on the left branch and $x_m$ a point on the middle branch of the critical manifold $\bar{C}_0$. Theorem 3.4 in \cite{KruSzm2} tells us that the canards are stable cycles if the function
\begin{equation*}
R(h)=\int_{x_l(h)}^{x_m(h)}\frac{\partial \bar{f}}{\partial x_1}(x_1,\phi(x_1))\frac{\phi'(x_1)}{\bar{g}(x_1,\phi(x_1))}dx_1
\end{equation*}
is negative for all values $h\in(0,\phi(\frac{\sqrt{91}}{15})]$ where $x_1=\frac{\sqrt{91}}{15}$ is the second fold point of $\bar{C}_0$ besides $x_1=0$. In our case we have
\begin{equation*}
R(h)=\int_{x_l(h)}^{x_m(h)}\frac{(\frac{\sqrt{91}}{5}x_1-3x_1^2)^2}{x-\frac{\sqrt{91}}{10}x_1^2+x_1^3}dx
\end{equation*}
with $x_l(h)\in[-\frac{\sqrt{91}}{30},0)$ and $x_m(h)\in(0,\frac{\sqrt{91}}{15}]$. Figure \ref{fig:canard_stability} shows a numerical plot of the function $R(h)$ for the relevant values of $h$ which confirms the required result.\\ 

\begin{figure}[htbp]
	\centering
		\includegraphics{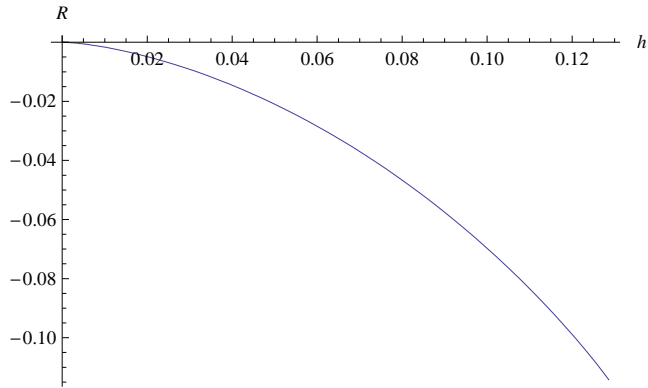}
	\caption{Plot of the function $R(h)$ for $h\in(0,\phi(\frac{\sqrt{91}}{15})]$.}
	\label{fig:canard_stability}
\end{figure}

\textit{Remark:} We have computed an explicit algebraic expression for $R'(h)$ with a computer algebra system. This expression yields $R'(h)<0$ for $h\in(0,\phi(\frac{\sqrt{91}}{15})]$, confirming that $R(h)$ is decreasing.

\end{proof}

As long as we stay on the critical manifold $C_0$ of the full system, the analysis of the bifurcations and geometry of (\ref{eq:red1}) give good approximations to the dynamics of the FitzHugh-Nagumo equation because the rescaling $x_2=\epsilon \bar{x}_2$ leaves the plane $x_2=0$ invariant. Next we use the dynamics of the $\bar{x}_2$-coordinate in system (\ref{eq:red2}) to obtain better insight into the dynamics  when $x_2\neq0$. The critical manifold $D_0$ of (\ref{eq:red2}) is:
\begin{equation*}
D_0=\{(x_1,\bar{x}_2)\in\mathbb{R}^2:s\bar{x}_2c'(x_1)=x_1-c(x_1)\}
\end{equation*}
We are interested in the geometry of the periodic orbits shown in Figure \ref{fig:subsysgeo} that emerge from the Hopf bifurcation at $p_{H,-}$. Observe that the amplitude of the orbits in the $x_1$ direction is much larger that than in the $x_2$-direction. Therefore we predict only a single small excursion in the $x_2$ direction for $p$ slightly larger than $p_{H,-}$ as shown in Figures \ref{fig:sub1} and \ref{fig:sub3}. The wave speed changes the amplitude of this $x_2$ excursion with a smaller wave speed implying a larger excursion. Hence equation (\ref{eq:red1}) is expected to be a very good approximation for periodic orbits in the FitzHugh-Nagumo equation with fast wave speeds. Furthermore the periodic orbits show two $x_2$ excursions in the relaxation regime after the canard explosion; see Figure \ref{fig:sub2}.

\begin{figure}[htbp]
 \subfigure[Small orbit near Hopf point ($p=0.058$, $s=1.37$)]{\includegraphics[width=0.3\textwidth]{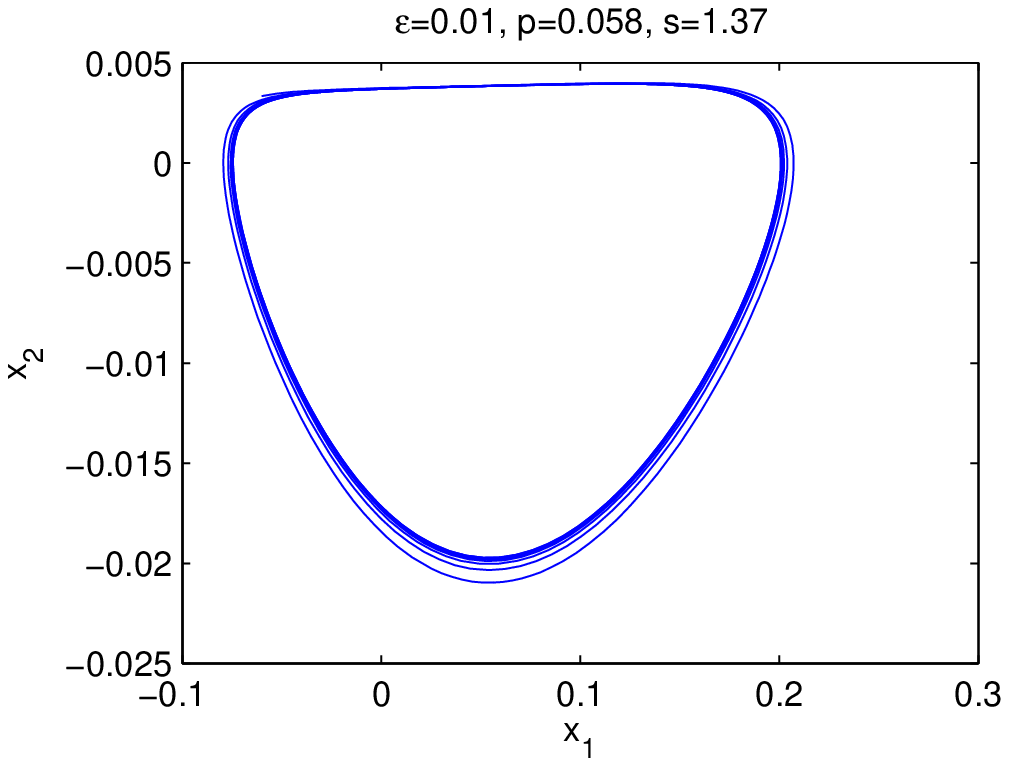} \label{fig:sub1}}
 \subfigure[Orbit after canard explosion ($p=0.06$, $s=1.37$)]{\includegraphics[width=0.3\textwidth]{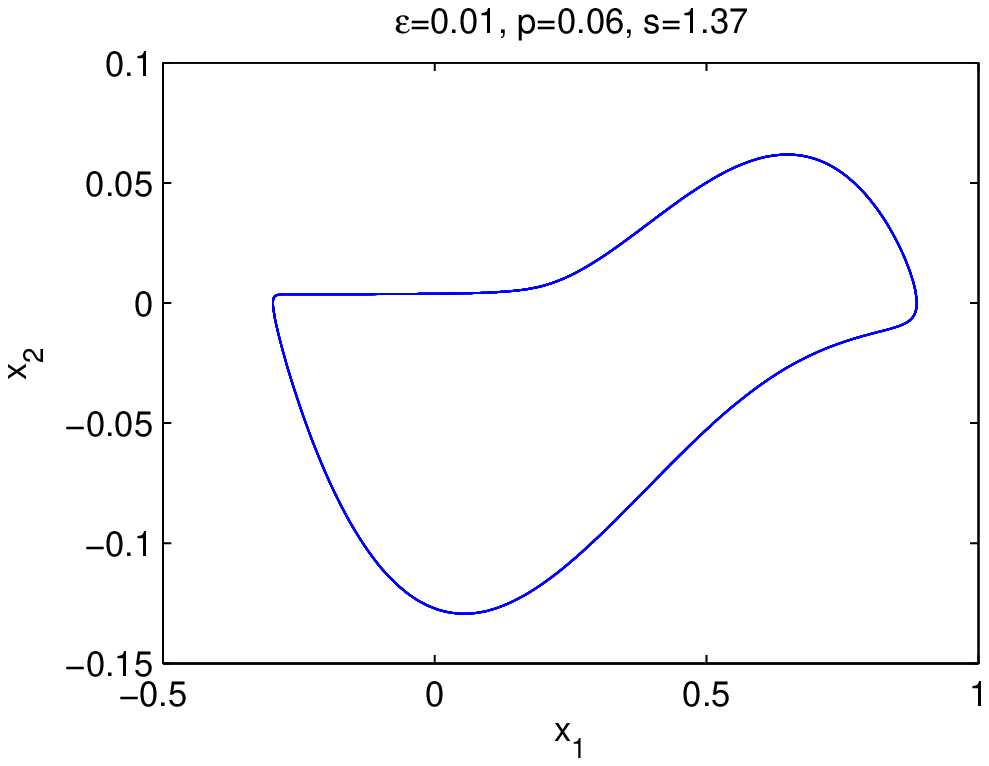}\label{fig:sub2}}
 \subfigure[Different wave speed ($p=0.058$, $s=0.2$)]{\includegraphics[width=0.3\textwidth]{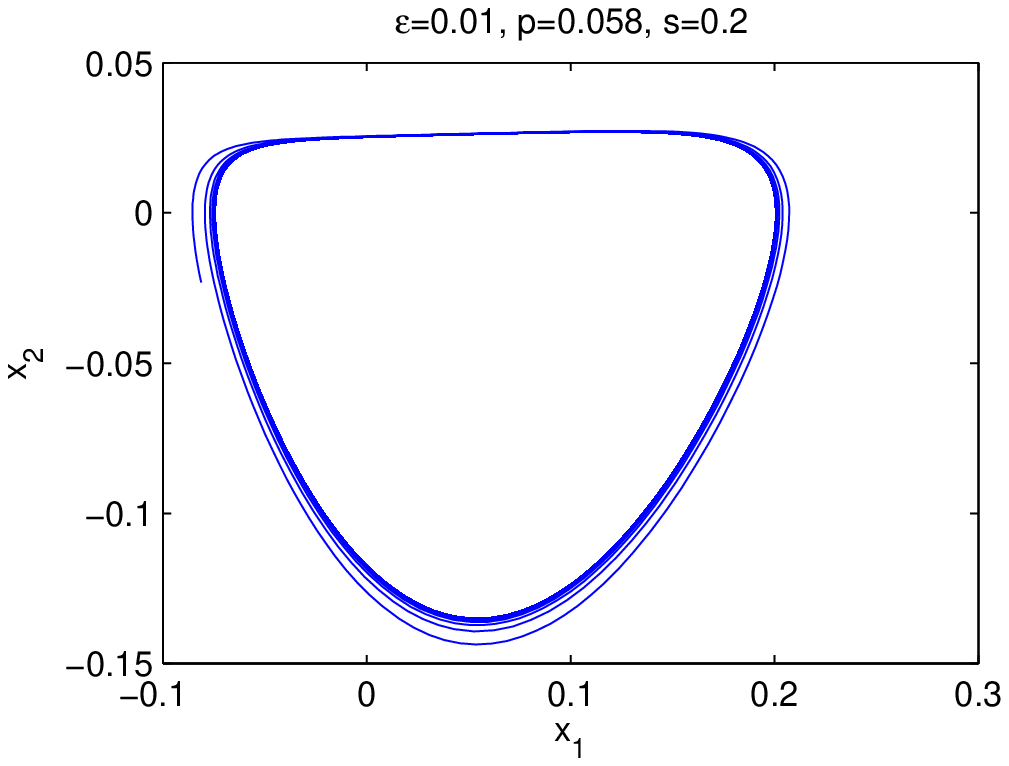}\label{fig:sub3}}
\caption{Geometry of periodic orbits in the $(x_1,x_2)$-variables of the 2-variable slow subsystem (\ref{eq:red2}). Note that here $x_2=\epsilon \bar{x}_2$ is shown. Orbits have been obtained by direct forward integration for $\epsilon=0.01$.}
\label{fig:subsysgeo}
\end{figure} 

\section{The Full System}
\subsection{Hopf Bifurcation}
\label{sec:hopf}
The characteristic polynomial of the linearization of the FitzHugh-Nagumo equation (\ref{eq:fhn}) at its unique equilibrium point is
\begin{equation*}
P(\lambda)=\frac{\epsilon }{5 s}+\left(-\frac{\epsilon }{s}-\lambda \right) \left(-\frac{1}{50}+\frac{11 x_1^*}{25}-\frac{3 (x_1^*)^2}{5}-\frac{s \lambda }{5}+\lambda ^2\right)
\end{equation*}
Denoting $P(\lambda)=c_0+c_1\lambda+c_2\lambda^2+c_3\lambda^3$, a necessary condition for $P$ to have pure imaginary roots is that $c_0 = c_1 c_2$. The solutions of this equation can be expressed parametrically as a curve $(p(x_1^*),s(x_1^*))$:
\begin{eqnarray}
\label{eq:solHopf}
s(x_1^*)^2 &=& \frac{50\epsilon(\epsilon - 1)}{1 + 10\epsilon -22 x_1^* +30 (x_1^*)^2} \nonumber \\
p(x_1^*) &=& (x_1^*)^3 - 1.1 (x_1^*)^2 + 1.1
\end{eqnarray}

\begin{proposition}
\label{prop:hopf}
In the singular limit $\epsilon\rightarrow 0$ the U-shaped bifurcation curves of the FitzHugh-Nagumo equation have vertical asymptotes given by the points $p_-\approx 0.0510636 $ and $p_+\approx 0.558418$ and a horizontal asymptote given by $\{(p,s):p\in[p_-,p_+]\quad\text{and}\quad s=0\}$. Note that at $p_\pm$ the equilibrium point passes through the two fold points.
\end{proposition}

\begin{proof}
The expression for $s(x_1^*)^2$ in (\ref{eq:solHopf}) is positive when
$1 + 10\epsilon -22 x_1^* + 30 (x_1^*)^2 < 0$. For values of $x_1^*$ 
between the roots of $1 -22 x_1^* + 30 (x_1^*)^2 = 0$, $s(x_1^*)^2 \to 0$ in (\ref{eq:solHopf}) as $\epsilon \to 0$. The values of $p_-$ and $p_+$ in
the proposition are approximations to the value of $p(x_1^*)$ in
(\ref{eq:solHopf}) at the roots of $1 -22 x_1^* + 30 (x_1^*)^2 = 0$.
As $\epsilon \to 0$, solutions of the equation $s(x_1^*)^2 = c > 0$ 
in (\ref{eq:solHopf}) yield values of $x_1^*$ that tend to one of the 
two roots of $1 - 22 x_1^* + 30 (x_1^*)^2 = 0$. 
The result follows.
\end{proof}

\begin{figure}[htbp]
 \subfigure[Projection onto $(x_1,y)$]{\includegraphics[width=0.46\textwidth]{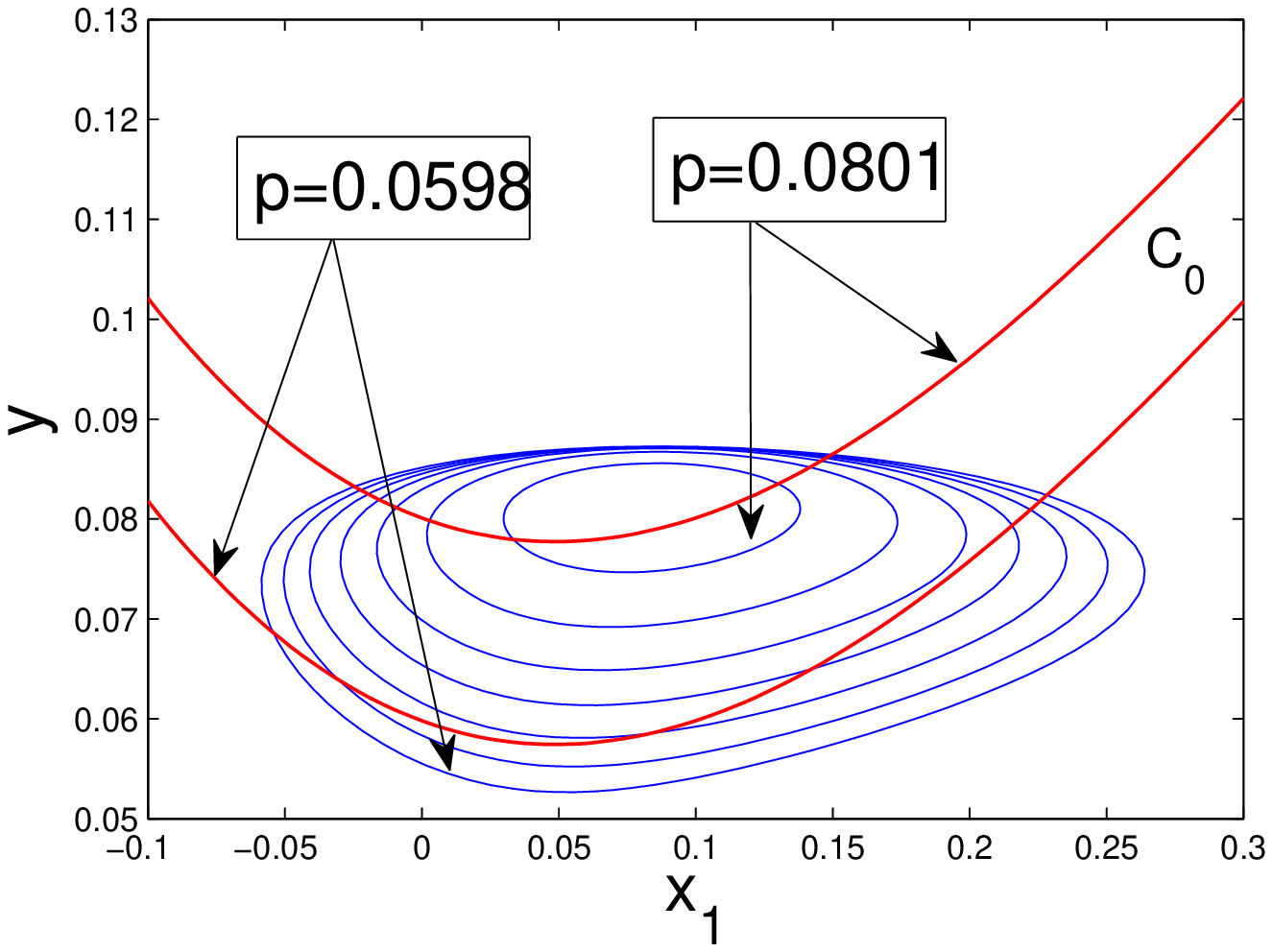}\label{fig:po1}}
 \subfigure[Projection onto $(x_1,x_2)$]{\includegraphics[width=0.46\textwidth]{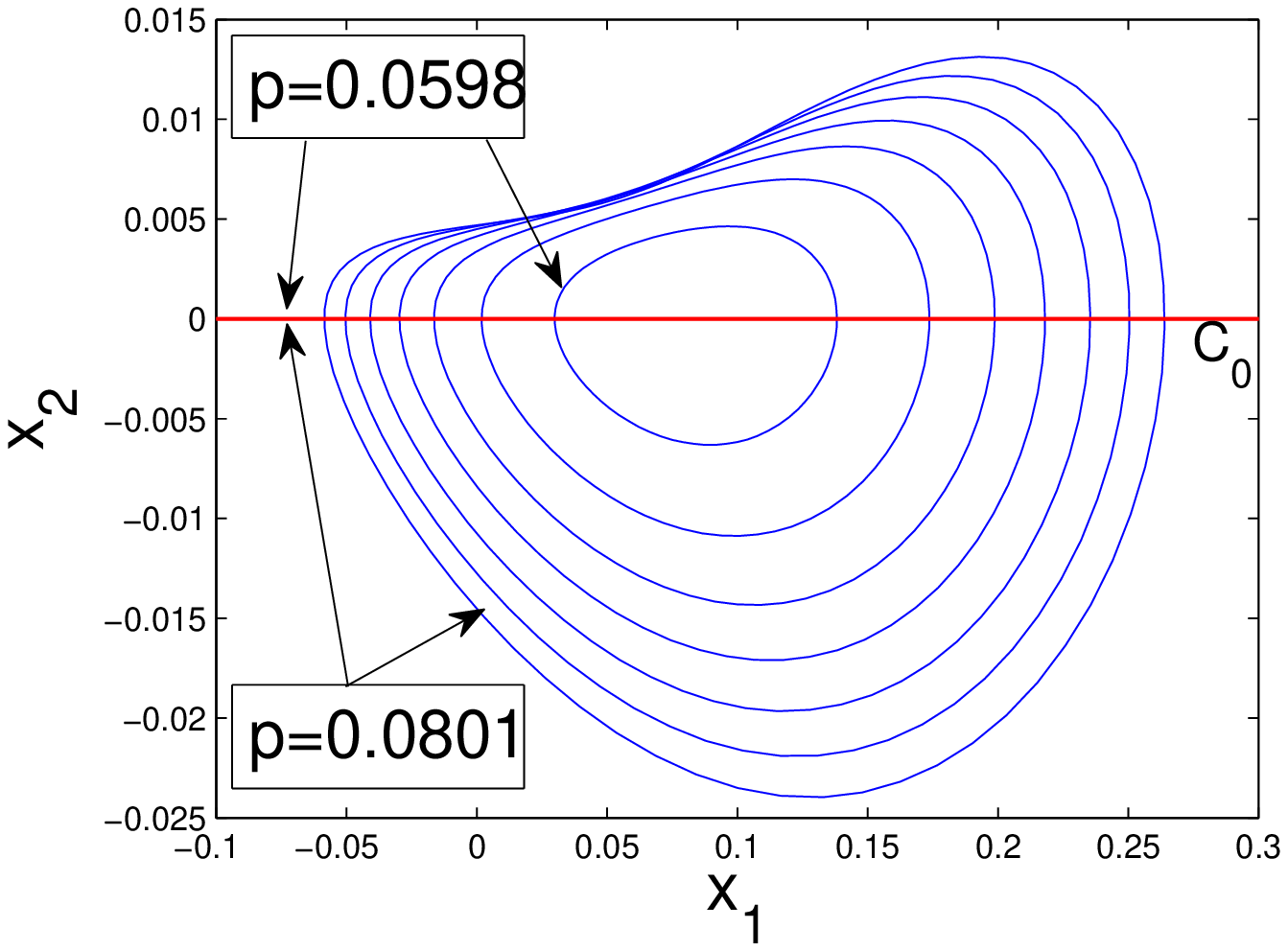}\label{fig:po2}}
\caption{Hopf bifurcation at $p\approx 0.083$, $s=1$ and $\epsilon=0.01$. The critical manifold $C_0$ is shown in red and periodic orbits are shown in blue. Only the first and the last critical manifold for the continuation run are shown; not all periodic orbits obtained during the continuation are displayed.}
\label{fig:porbits}
\end{figure}

The analysis of the slow subsystems  (\ref{eq:red1}) and (\ref{eq:red2}) gives a conjecture about the shape of the periodic orbits in the FitzHugh-Nagumo equation. Consider the parameter regime close to a Hopf bifurcation point. From (\ref{eq:red1}) we expect one part of the small periodic orbits generated in the Hopf bifurcation to lie close to the slow manifolds $C_{l,\epsilon}$ and $C_{m,\epsilon}$. Using the results about equation (\ref{eq:red2}) we anticipate the second part to consist of an excursion in the $x_2$ direction whose length is governed by the wave speed $s$. Figure \ref{fig:porbits} shows a numerical continuation in MatCont \cite{MatCont} of the periodic orbits generated in a Hopf bifurcation and confirms the singular limit analysis for small amplitude orbits. \\

Furthermore we observe from comparison of the $x_1$ and $x_2$ coordinates of the periodic orbits in Figure \ref{fig:po2} that orbits tend to lie close to the plane defined by $x_2=0$. More precisely, the $x_2$ diameter of the periodic orbits is observed to be $O(\epsilon)$ in this case. This indicates that the rescaling of Section \ref{sec:2slow1fast} can help to describe the system close to the U-shaped Hopf curve. Note that it is difficult to check whether this observation of an $O(\epsilon)$-diameter in the $x_2$-coordinate persists for values of $\epsilon<0.01$ since numerical continuation of canard-type periodic orbits is difficult to use for smaller $\epsilon$.\\

\begin{figure}[htbp]
 \subfigure[$GH^\epsilon_{1}$]{\includegraphics[width=0.46\textwidth]{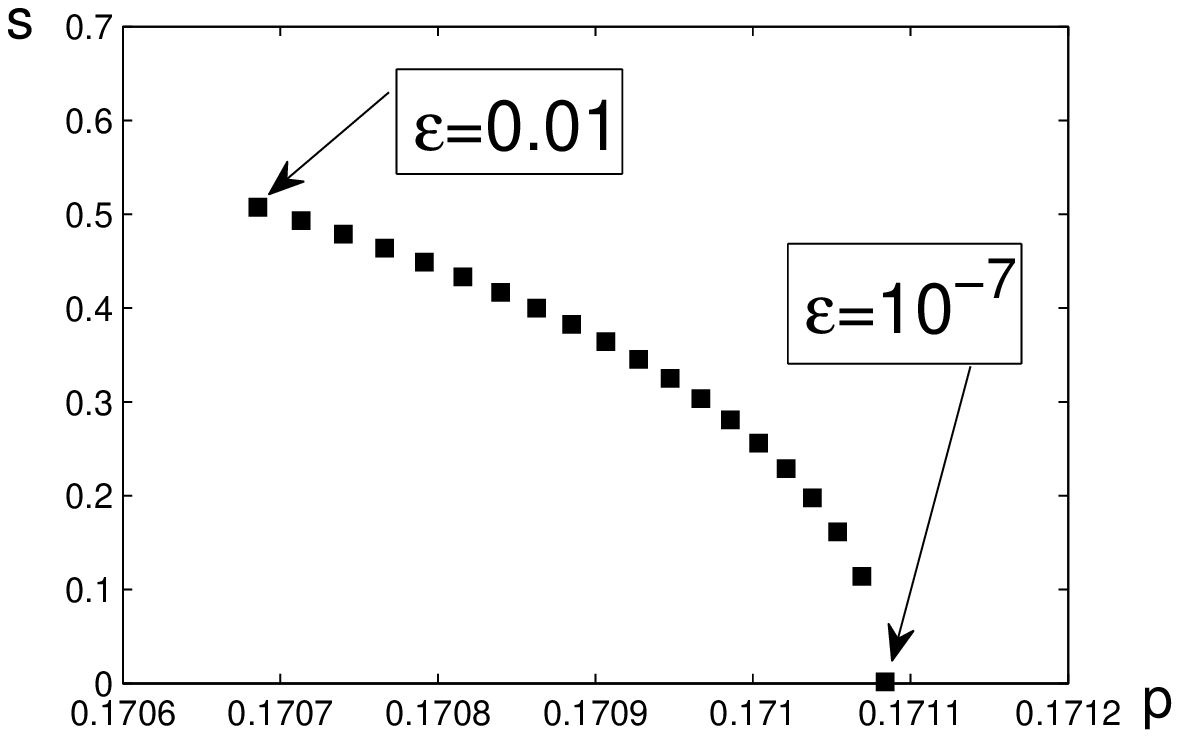}\label{fig:GH2}}
 \subfigure[$GH^\epsilon_{2}$]{\includegraphics[width=0.46\textwidth]{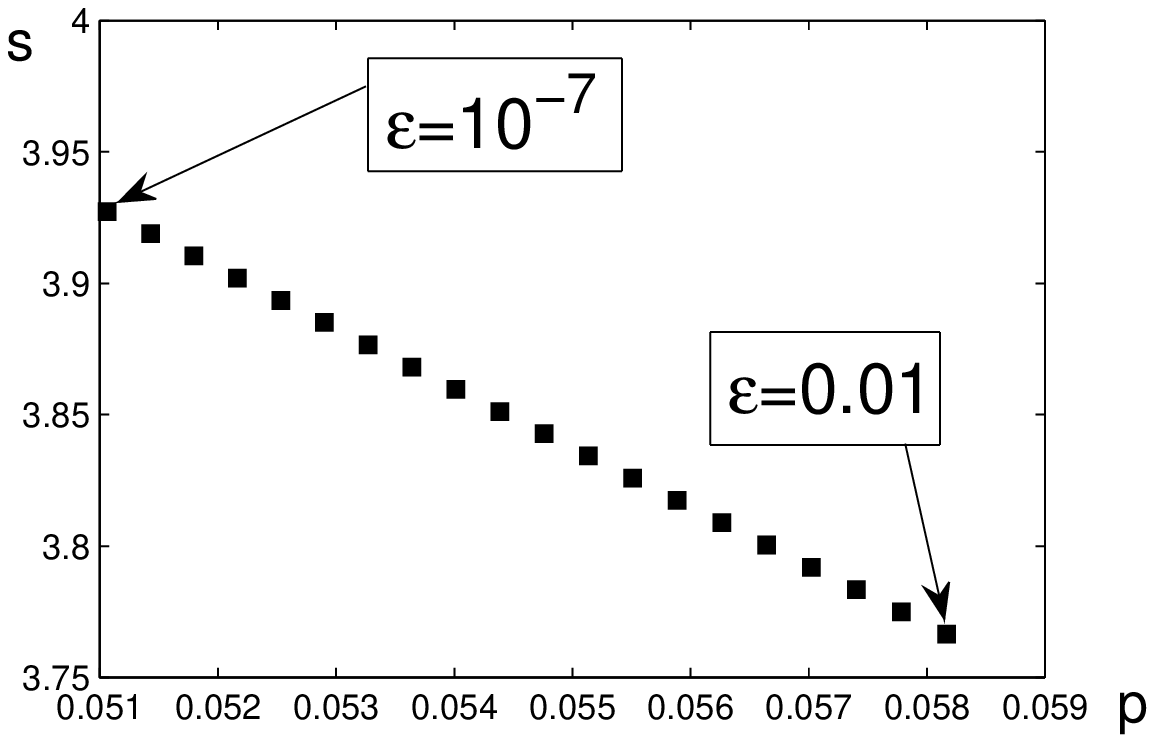}\label{fig:GH1}}
\caption{Tracking of two generalized Hopf points (GH) in $(p,s,\epsilon)$-parameter space. Each point in the figure corresponds to a different value of $\epsilon$. The point $GH^\epsilon_1$ in \ref{fig:GH2} corresponds to the point shown as a square in Figure \ref{fig:cusystem} and the point $GH^\epsilon_2$ in \ref{fig:GH1} is further up on the left branch of the U-curve and is not displayed in Figure \ref{fig:cusystem}.}
\label{fig:GH}
\end{figure}

In contrast to this, it is easily possible to compute the U-shaped Hopf curve using numerical continuation for very small values of $\epsilon$. We have used this possibility to track two generalized Hopf bifurcation points in three parameters $(p,s,\epsilon)$. The U-shaped Hopf curve has been computed by numerical continuation for a mesh of parameter values for $\epsilon$ between $10^{-2}$ and $10^{-7}$ using MatCont \cite{MatCont}. The two generalized Hopf points $GH^\epsilon_{1,2}$ on the left half of the U-curve were detected as codimension two points during each continuation run. The results of this ``three-parameter continuation'' are shown in Figure \ref{fig:GH}.\\ 

The two generalized Hopf points depend on $\epsilon$ and we find that their singular limits in $(p,s)$-parameter space are approximately:
\begin{equation*}
GH^0_1\approx (p=0.171,s=0)\qquad \text{and} \qquad GH^0_2\approx (p=0.051,s= 3.927)
\end{equation*}
We have not found a way to recover these special points from the fast-slow decomposition of the system. This suggests that codimension two bifurcations are generally diffcult to recover from the singular limit of fast-slow systems.\\

Furthermore the Hopf bifurcations for the full system on the left half of the U-curve are subcritical between $GH^\epsilon_1$ and $GH^\epsilon_2$ and supercritical otherwise. For the transformed system \eqref{eq:fhnscale} with two slow and one fast variable we observed that in the singular limit \eqref{eq:red1} for $\epsilon^2\rightarrow 0$ the Hopf bifurcation is supercritical. In the case of $\epsilon=0.01$ the periodic orbits for \eqref{eq:fhn} and \eqref{eq:red1} exist in overlapping regions for the parameter $p$ between the $p$-values of $GH^{0.01}_1$ and $GH^{0.01}_2$. This result indicates that \eqref{eq:fhnscale} can be used to describe periodic orbits that will interact with the homoclinic C-curve. 

\subsection{Homoclinic Orbits}
\label{sec:hetfull}

In the following discussion we refer to ``the'' C-shaped curve of homoclinic bifurcations of system~\eqref{eq:fhn_temp} as the parameters yielding a ``single-pulse'' homoclinic orbit. The literature as described in Section \ref{sec:fhn} shows that close to single-pulse homoclinic orbits we can expect multi-pulse homoclinic orbits that return close to the equilibrium point multiple times. Since the separation of slow manifolds $C_{\cdot,\epsilon}$ is exponentially small, homoclinic orbits of different types will always occur in exponentially thin bundles in parameter space. Values of $\epsilon< 0.005$ are small enough that the parameter region containing all the homoclinic orbits will be indistinguishable numerically from ``the'' C-curve that we locate. \\ 

The history of proofs of the existence of homoclinic orbits in the FitzHugh-Nagumo equation is quite extensive. The main step in their construction is the existence of a ``singular'' homoclinic orbit $\gamma_0$. We consider the case when the fast subsystem has three equilibrium points which we denote by $x_l\in C_l$, $x_m\in C_m$ and $x_r\in C_r$. Recall that $x_l$ coincides with the unique equilibrium $q=(x_1^*,0,x_1^*)$ of the full system for $p<p_-$. A singular homoclinic orbit is always constructed by first following the unstable manifold of $x_l$ in the fast subsystem given by $y=x_1^*$.\\

\begin{figure}[htbp]
\centering
 \includegraphics[width=0.8\textwidth]{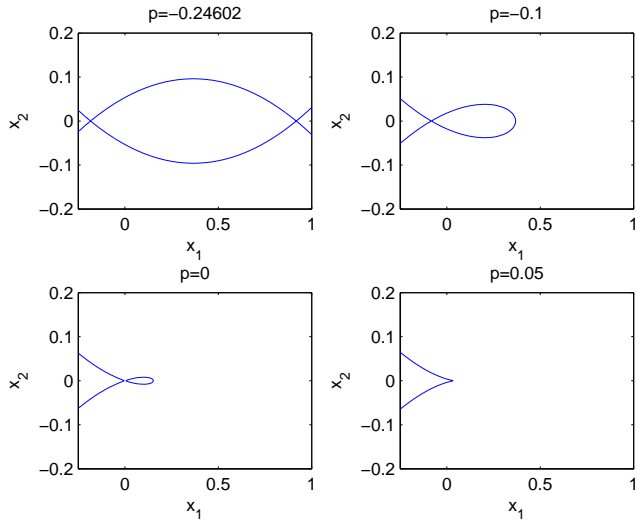}  
 \caption{Homoclinic orbits as level curves of $H(x_1,x_2)$ for equation (\ref{eq:fss2}) with $y=x_1^*$.}
 \label{fig:homs0}
\end{figure}

First assume that $s=0$. In this case the Hamiltonian structure - see Section \ref{sec:hetsub} and equation (\ref{eq:fss2}) - can be used to show the existence of a singular homoclinic orbit. Figure \ref{fig:homs0} shows level curves $H(x_1,x_2)=H(x_1^*,0)$ for various values of $p$. The double heteroclinic connection can be calculated directly using Proposition \ref{prop:doublehet} and solving $x_1^*+\bar{p}^*=p$ for $p$.

\begin{proposition}
There exists a singular double heteroclinic connection in the FitzHugh-Nagumo equation for $s=0$ and $p\approx -0.246016=p^*$.
\end{proposition}   

Techniques developed in \cite{Szmolyan1} show that the singular homoclinic orbits existing for $s=0$ and $p\in(p^*,p_-)$ must persist for perturbations of small positive wave speed and sufficiently small $\epsilon$. These orbits are associated to the lower branch of the C-curve. The expected geometry of the orbits is indicated by their shape in the singular limit shown  in Figure \ref{fig:homs0}. The double heteroclinic connection is the boundary case between the upper and lower half of the C-curve. It remains to analyze the singular limit for the upper half. In this case, a singular homoclinic orbit is again formed by following the unstable manifold of $x_l$ when it coincides with the equilibrium $q=(x_1^*,0,x_1^*)$ but now we check whether it forms a heteroclinic orbit with the stable manifold of $x_r$. Then we follow the slow flow on $C_r$ and return on a heteroclinic connection to $C_l$ for a different y-coordinate with $y>x_1^*$ and $y<c(x_{1,+})=f(x_1)+p$. From there we connect back via the slow flow. Using the numerical method described in Section \ref{sec:hetsub} we first set $y=x_1^*$; note that the location of $q$ depends on the value of the parameter $p$. The task is to check when the system 
\begin{eqnarray}
\label{eq:hetateq}
x_1'&=& x_2\nonumber\\
x_2'&=& \frac{1}{5}\left(f(x_1)+y-p\right)
\end{eqnarray} 
has heteroclinic orbits from $C_l$ to $C_r$ with $y = x_1^*$. The result of this computation is shown in Figure \ref{fig:het_full} as the red curve. We have truncated the result at $p=-0.01$. In fact, the curve in Figure \ref{fig:het_full} can be extended to $p=p^*$. Obviously we should view this curve as an approximation to the upper part of the C-curve.\\
   
\begin{figure}[htbp]
\centering
\psfrag{p}{$p$}
\psfrag{s}{$s$}
 \includegraphics[width=0.5\textwidth]{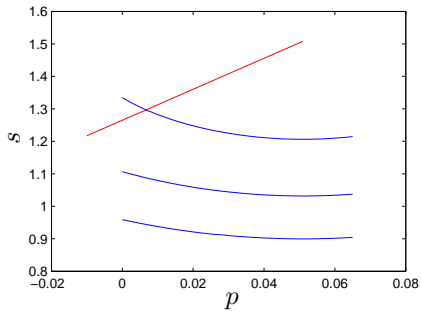} 
 \caption{\label{fig:het_full}Heteroclinic connections for equation \eqref{eq:hetateq} in parameter space. The red curve indicates left-to-right connections for $y=x_1^*$ and the blue curves indicate right-to-left connections for $y=x_1^*+v$ with $v=0.125,0.12,0.115$ (from top to bottom).}
\end{figure}

If the connection from $C_r$ back to $C_l$ occurs with vertical coordinate $x_1^*+v$, it is a trajectory of 
system (\ref{eq:hetateq}) with $y=x_1^*+v$. Figure \ref{fig:het_full} shows values of $(p,s)$ at which these
heteroclinic orbits exist for $v=0.125,0.12,0.115$. An intersection between a red and a blue curve indicates a singular homoclinic orbit. Further computations show that increasing the value of $v$ slowly beyond $0.125$ yields intersections everywhere along the red curve in Figure \ref{fig:het_full}. Thus the values of $v$ on the homoclinic orbits are expected to grow as $s$ increases along the upper branch of the C-curve. Since there cannot be any singular homoclinic orbits for $p\in (p_-,p_+)$ we have to find the intersection of the red curve in Figure \ref{fig:het_full} with the vertical line $p=p_-$. Using the numerical method to detect heteroclinic connections gives:

\begin{proposition}
The singular homoclinic curve for positive wave speed terminates at $p=p_-$ and $s\approx 1.50815=s^*$ on the right and at $p=p^*$ and $s=0$ on the left. 
\end{proposition}  

In $(p,s)$-parameter space define the points:
\begin{equation}
\label{eq:mpts}
A=(p^*,0),\qquad B=(p_-,0),\qquad C=(p_-,s^*)
\end{equation}
In Figure \ref{fig:amazing} we have computed the homoclinic C-curve for values of $\epsilon$ between $10^{-2}$ and $5\cdot 10^{-5}$. Together with the singular limit analysis above, this yields strong numerical evidence for the following conjecture:

\begin{conjecture}
\label{cjt:hom}
The C-shaped homoclinic bifurcation curves converge to the union of the segments $AB$ and $AC$ as $\epsilon \rightarrow 0$.
\end{conjecture}  

\textit{Remark 1:} Figure 4 of Krupa, Sandstede and Szmolyan \cite{KSS1997} shows a ``wedge'' that resembles shown in Figure \ref{fig:amazing}. The system that they study sets $p=0$ and 
varies $a$ with $a\approx 1/2$. For $a=1/2$ and $p=0$, the equilibrium point $q$ is located at the origin and the fast subsystem with $y=0$ has a double heteroclinic connection at $q$ to the saddle equilibrium $(1,0,0)\in C_r$. The techniques developed in \cite{KSS1997} use this double heteroclinic connection as a starting point. Generalizations of the results in \cite{KSS1997} might provide a strategy to prove Conjecture \ref{cjt:hom} rigorously, a possibility that we have not yet considered. However, we think that 1-homoclinic orbits in the regime we study come in pairs and that the surface of 1-homoclinic orbits in $(p,s,\epsilon)$ space differs qualitatively from that described by Krupa, Sandstede and Szmolyan.\\

\begin{figure}[htbp]
\centering
 \includegraphics[width=0.9\textwidth]{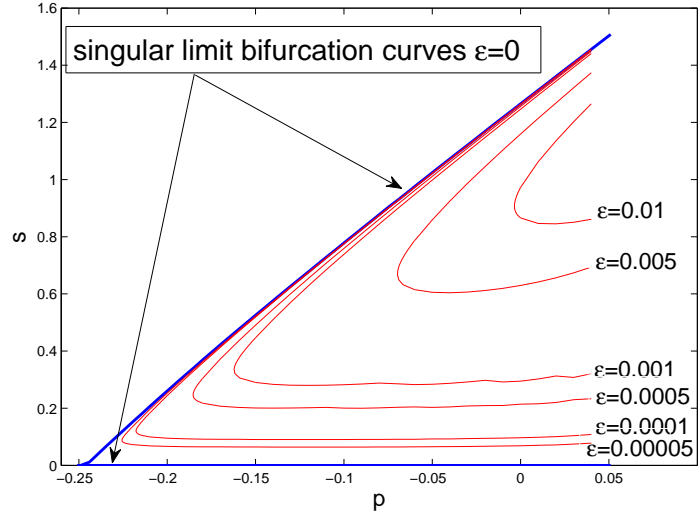}  
 \caption{\label{fig:amazing}Singular limit ($\epsilon=0$) of the C-curve is shown in blue and parts of several C-curves for $\epsilon>0$ have been computed (red).}
\end{figure}

\textit{Remark 2:}  We have investigated the termination or turning mechanism of the C-curve at its upper end. The termination points shown in Figure \ref{fig:cusystem} have been obtained by a different geometric method. It relies on the observation that, in addition to the two fast heteroclinic connections, we have to connect near $C_l$ back to the equilibrium point $q$ to form a homoclinic orbit; the two heteroclinic connections might persist as intersections of suitable invariant manifolds but we also have to investigate how the flow near $C_{l,\epsilon}$ interacts with the stable manifold $W^s(q)$. These results will be reported elsewhere, but we note here that $p_{turn}(\epsilon)\rightarrow p_-$.\\ 

The numerical calculations of the C-curves for $\epsilon\leq 10^{-3}$ are new. Numerical continuation using the boundary value methods implemented in AUTO \cite{Doedel_AUTO2000} or MatCont \cite{MatCont} becomes very difficult for these small values of $\epsilon$ \cite{Sneydetal}. Even computing with values $\epsilon=O(10^{-2})$ using boundary value methods is a numerically challenging problem. The method we have used does not compute the homoclinic orbits themselves while it locates the homoclinic C-curve accurately in parameter space. To motivate our approach consider Figure \ref{fig:splitting} which shows the unstable manifold $W^u(q)$ for different values of $s$ and fixed $p$. We observe that homoclinic orbits can only exist near two different wave speeds $s_1$ and $s_2$ which define the parameters where $W^u(q)\subset W^s(C_{l,\epsilon})$ or $W^u(q)\subset W^s(C_{r,\epsilon})$. Figure \ref{fig:splitting} displays how $W^u(q)$ changes as $s$ varies for the fixed value $p=0.05$. If $s$ differs from the points $s_1$ and $s_2$ that define the lower and upper branches of the C-curve for the given value of $p$, then $|x_1|$ increases rapidly on $W^u(q)$ away from $q$. The changes in sign of $x_1$ on $W^u(q)$ identify values of $s$ with homoclinic orbits. The two splitting points that mark these sign changes are visible in Figure \ref{fig:splitting}. Since trajectories close to the slow manifolds separate exponentially away from them, we are able to assess the sign of $x_1$ unambiguously on trajectories close to the slow manifold and find small intervals $(p,s_1\pm 10^{-15})$ and $(p,s_2\pm 10^{-15})$ that contain the values of $s$ for which there are homoclinic orbits.\\  

\begin{figure}[htbp]
 \subfigure[$\epsilon=0.01$, $p=0.05$, \text{$s\in[0.1,0.9]$}]{\includegraphics[width=0.45\textwidth]{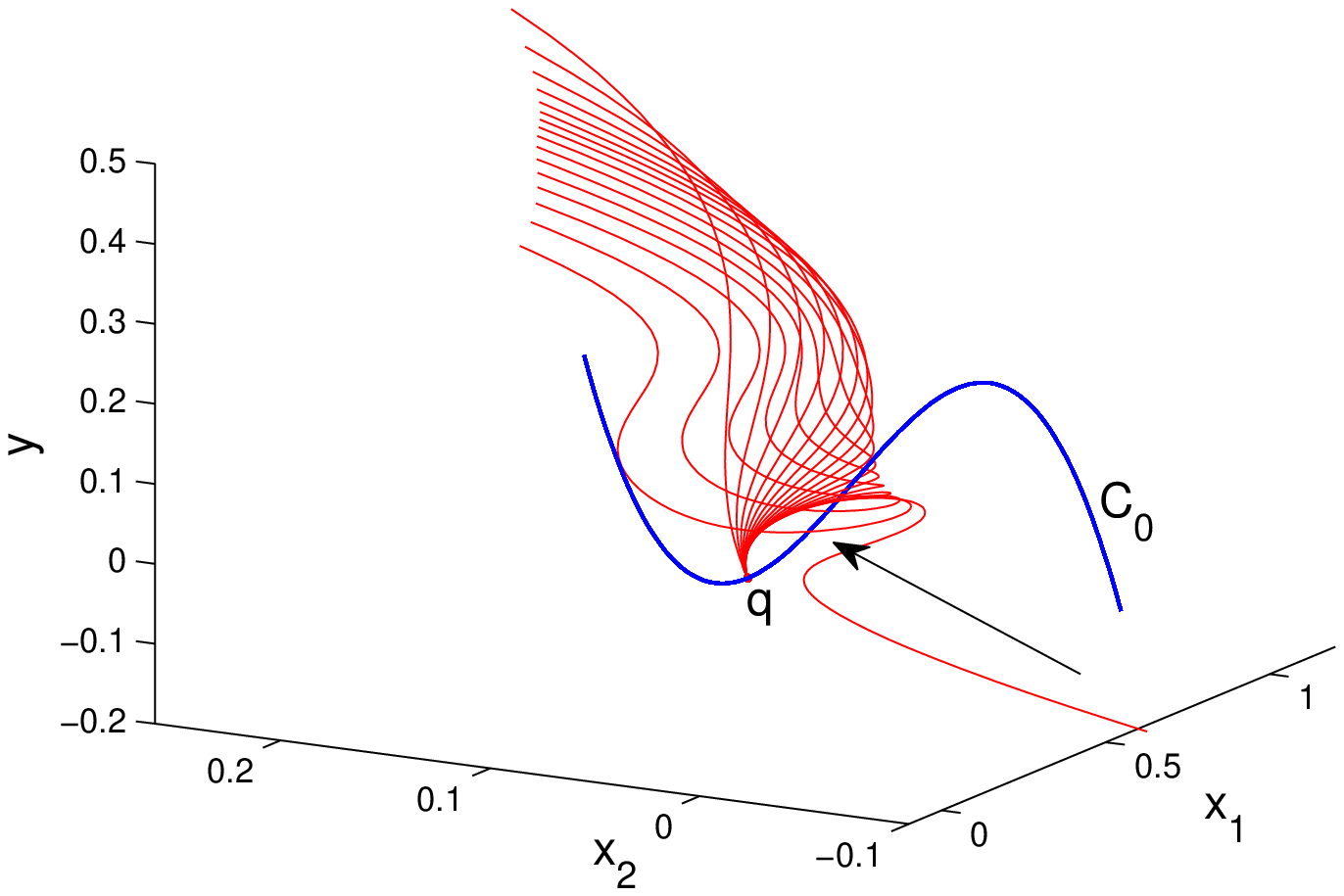} \label{fig:split1}}
 \subfigure[$\epsilon=0.01$, $p=0.05$, \text{$s\in[0.9,1.5]$}]{\includegraphics[width=0.45\textwidth]{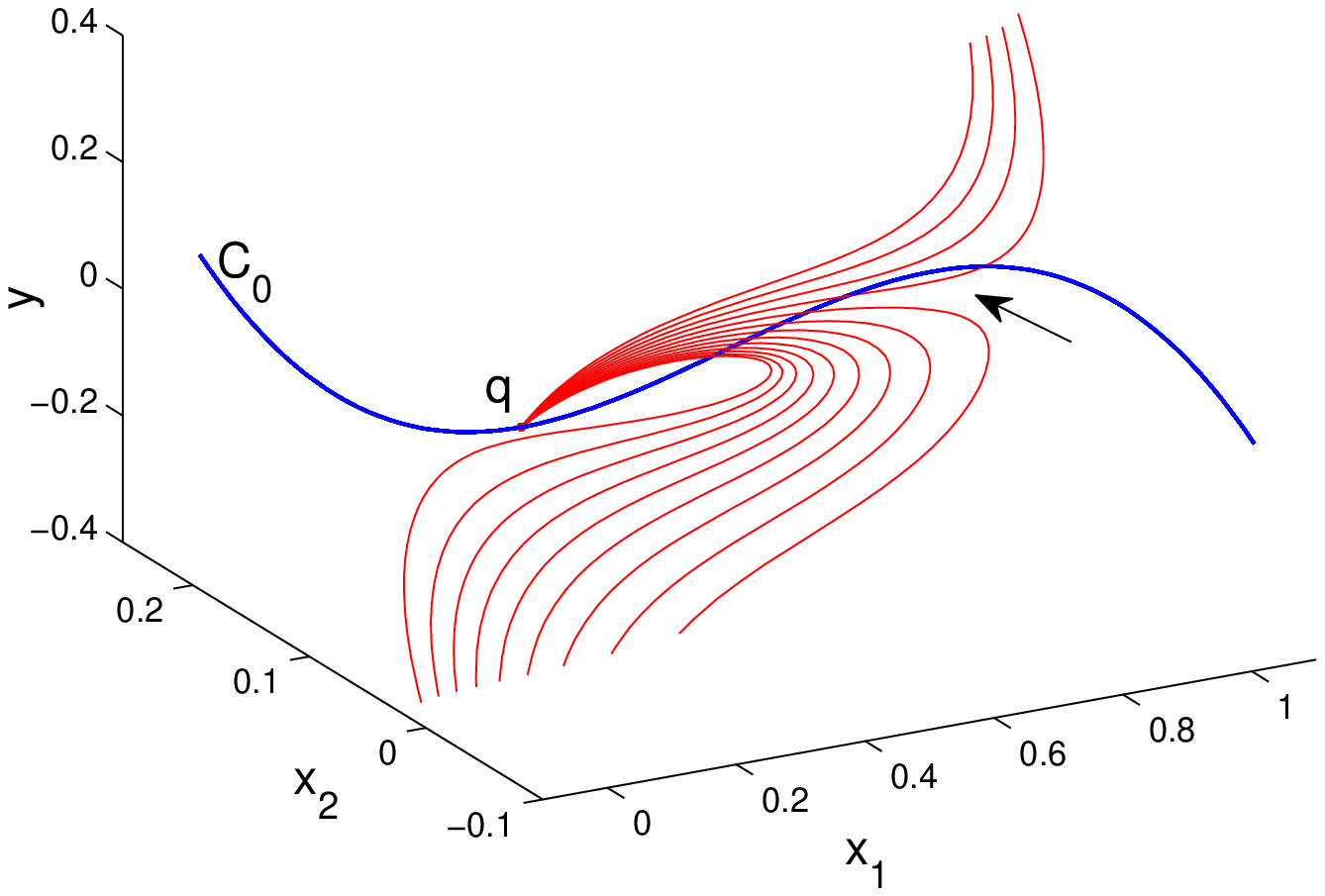} \label{fig:split2}}
\caption{\label{fig:splitting}Strong ``splitting'', marked by an arrow, of the unstable manifold $W^u(q)$ (red) used in the calculation of the homoclinic C-curve for small values of $\epsilon$. The critical manifold $C_0$ is shown in blue. The spacing in $s$ is $0.05$ for both figures.}
\end{figure}

The geometry of the orbits along the upper branch of the C-curve is obtained by approximating it with two fast singular heteroclinic connections and parts of the slow manifolds $C_{r,\epsilon}$ and $C_{l,\epsilon}$; this process has been described several times in the literature when different methods were used to prove the existence of ``fast waves'' (see e.g. \cite{Hastings1,Carpenter,JonesKopellLanger}). 

\section{Conclusions}

\begin{sidewaysfigure}
\centering
\psfrag{x1}{$x_1$}
\psfrag{x2}{$x_2$}
\psfrag{y}{$y$}
\psfrag{p}{$p$}
\psfrag{s}{$s$}
\psfrag{A}{$A$}
\psfrag{B}{$B$}
\psfrag{C}{$C$}
\psfrag{canards}{Canards in equation (\ref{eq:red1}),(\ref{eq:red2})}
\psfrag{Ccurve}{C-curve}
\psfrag{Ucurve}{U-curve}
\psfrag{slowflowbif}{slow flow bifurcation $p=p_-$}
\psfrag{Hopfred1red2}{Hopf bif. $p_{H,-}$ in (\ref{eq:red1}),(\ref{eq:red2})}
 \includegraphics[width=0.55\textwidth]{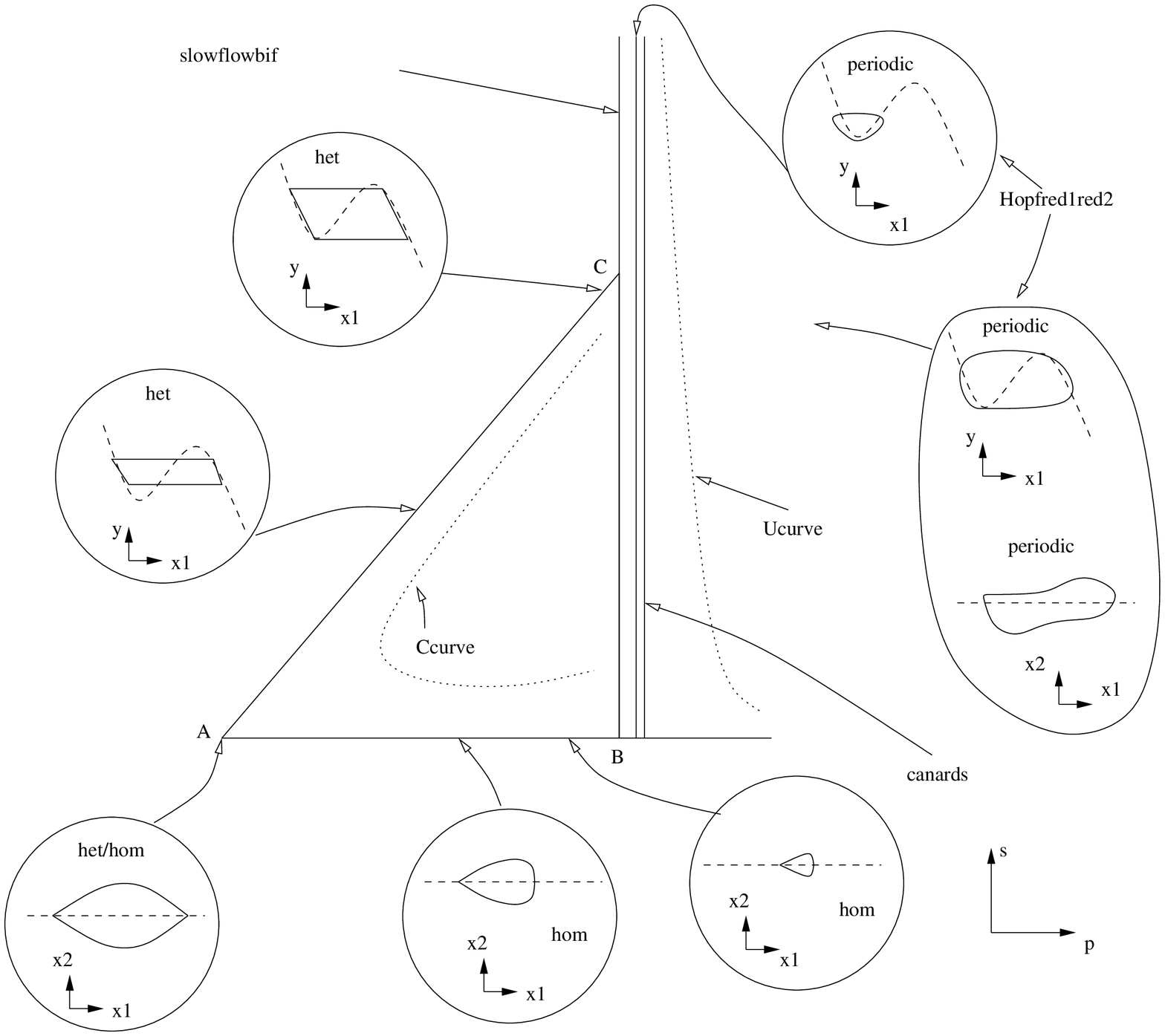}  
 \caption{Sketch of the singular bifurcation diagram for the FitzHugh-Nagumo equation (\ref{eq:fhn}). The points $A,B$ and $C$ are defined in (\ref{eq:mpts}). The part of the diagram obtained from equations (\ref{eq:red1}),(\ref{eq:red2}) corresponds to the case ``$\epsilon^2=0$ and $\epsilon\neq0$ and small''. In this scenario the canards to the right of $p=p_-$ are stable (see Proposition \ref{prop:stablecanards}). The phase portrait in the upper right for equation (\ref{eq:red1}) shows the geometry of a small periodic orbit generated in the Hopf bifurcation of (\ref{eq:red1}). The two phase portraits below it show the geometry of these periodic orbits further away from the Hopf bifurcation for (\ref{eq:red1}),(\ref{eq:red2}). Excursions of the periodic orbits/canards for $p>p_-$ decrease for larger values of $s$. Note also that we have indicated as dotted lines the C-curve and the U-curve for positive $\epsilon$ to allow a qualitative comparison with Figure \ref{fig:cusystem}.}
 \label{fig:singbif}
\end{sidewaysfigure}

Our results are summarized in the singular bifurcation diagram shown in Figure \ref{fig:singbif}.
This figure shows information obtained by a combination of fast-slow decompositions, classical dynamical systems techniques and robust numerical algorithms that work for very small values of $\epsilon$. It recovers and extends to smaller values of $\epsilon$ the CU-structure described in \cite{Sneydetal} for the FitzHugh-Nagumo equation. The U-shaped Hopf curve was computed with an explicit formula, and the homoclinic C-curve was determined by locating transitions between different dynamical behaviors separated by the homoclinic orbits. All the results shown as solid lines in Figure \ref{fig:singbif} have been obtained by considering a singular limit. The lines $AB$ and $AC$ as well as the slow flow bifurcation follow from the singular limit $\epsilon\rightarrow 0$ yielding the fast and slow subsystems of the FitzHugh-Nagumo equation \eqref{eq:fhn}. The analysis of canards and periodic orbits have been obtained from equations (\ref{eq:red1}) and (\ref{eq:red2}) where the singular limit $\epsilon^2\rightarrow 0$ was used (see Section \ref{sec:2slow1fast}). We have also shown the C- and U-curves in Figure \ref{fig:singbif} as dotted lines to orient the reader how the results from Proposition \ref{prop:hopf} and Conjecture \ref{cjt:hom} fit in.\\ 

We also observed that several dynamical phenomena are difficult to recover from the singular limit fast-slow decomposition. In particular, the codimension two generalized Hopf bifurcation does not seem to be observable from the singular limit analysis. Furthermore the homoclinic orbits can be constructed from the singular limits but it cannot be determined directly from the fast and slow subsystems that they are of Shilnikov-type.\\
  
The type of analysis pursued here seems to be very useful for other multiple time-scale problems involving multi-parameter bifurcation problems. In future work, we shall give a geometric analysis of the folding/turning mechanism of the homoclinic C-curve, a feature of this system we have not been able to determine directly from our singular limit analysis. That work relies upon new methods for calculating $C_{l,\epsilon}$ and $C_{r,\epsilon}$ which are invariant slow manifolds of ``saddle-type'' with both stable and unstable manifolds.\\

We end with brief historical remarks. The references cited in this paper discuss mathematical challenges posed by the FitzHugh-Nagumo equation, how these challenges have been analyzed and their relationship to general questions about multiple time-scale systems. Along the line $AB$ in Figure \ref{fig:singbif} we encounter a perturbation problem regarding the persistence of homoclinic orbits that can be solved using Fenichel theory \cite{Szmolyan1}. The point $A$ marks the connection between fast and slow waves in $(p,s)$-parameter space which has been investigated in $(\epsilon,s)$-parameter space in \cite{KSS1997}. We view this codimension 2 connectivity as one of the key features of the FitzHugh-Nagumo system. The perturbation problem for homoclinic orbits close to the line $AC$ was solved using several methods and was put into the context of multiple time-scale systems in \cite{JonesKopell,JonesKopellLanger}, where the Exchange Lemma overcame difficulties in tracking $W^u(q)$ when it starts jumping away from $C_{r,\epsilon}$. This theory provides rigorous foundations that support our numerical computations and their interpretation.\\

{\bf Acknowledgment:} This research was partially supported by the National Science Foundation and Department of Energy.

\bibliographystyle{plain}

\end{document}